\documentclass[]{ieeeconf}      

\IEEEoverridecommandlockouts

\def\baselinestretch{0.987}

\makeatletter
\let\IEEEproof\proof

\let\proof\@undefined
\let\endproof\@undefined
\makeatother

\usepackage{multicol}
\usepackage[bookmarks=false]{hyperref}
\usepackage{graphicx}
\usepackage{latexsym}
\usepackage{amsmath,amssymb,subfigure,amsthm}
\usepackage{epstopdf}
\usepackage[ruled, vlined]{algorithm2e}

\newcommand{\mc}[1]{\mathcal{#1}}

\newcommand{\bea}{\begin{eqnarray}}
\newcommand{\eea}{\end{eqnarray}}
\newcommand{\beas}{\begin{eqnarray*}}
\newcommand{\eeas}{\end{eqnarray*}}
\newcommand{\leftm}{\left[\begin{array}}
\newcommand{\rightm}{\end{array}\right]}
\newcommand{\reals}{\mbox{$\mathbb R$}}
\newcommand{\ones}{\textbf{1}}

\newcommand{\R}{{T }}
\newtheorem{thm}{Theorem}[section]
\newtheorem{cor}[thm]{Corollary}
\newtheorem{prop}[thm]{Proposition}

\newtheorem{definition}[thm]{Definition}
\newtheorem{rem}[thm]{Remark}

\def\span{{\mbox{\bf span}}}

\def\diag{{\mbox{\bf diag}}}
\def\rk{{\mbox{\bf rk}}}

\graphicspath{{figures/}}

\usepackage{Cris}

\newcommand\AF[1]{}

\definecolor{magenta}{rgb}{1,0,1}
\newcommand{\magenta}[1]{{ #1}}
\newcommand\AFmod[2]{\magenta{#1}}

\def\trajA{{\includegraphics[width=0.6\columnwidth]{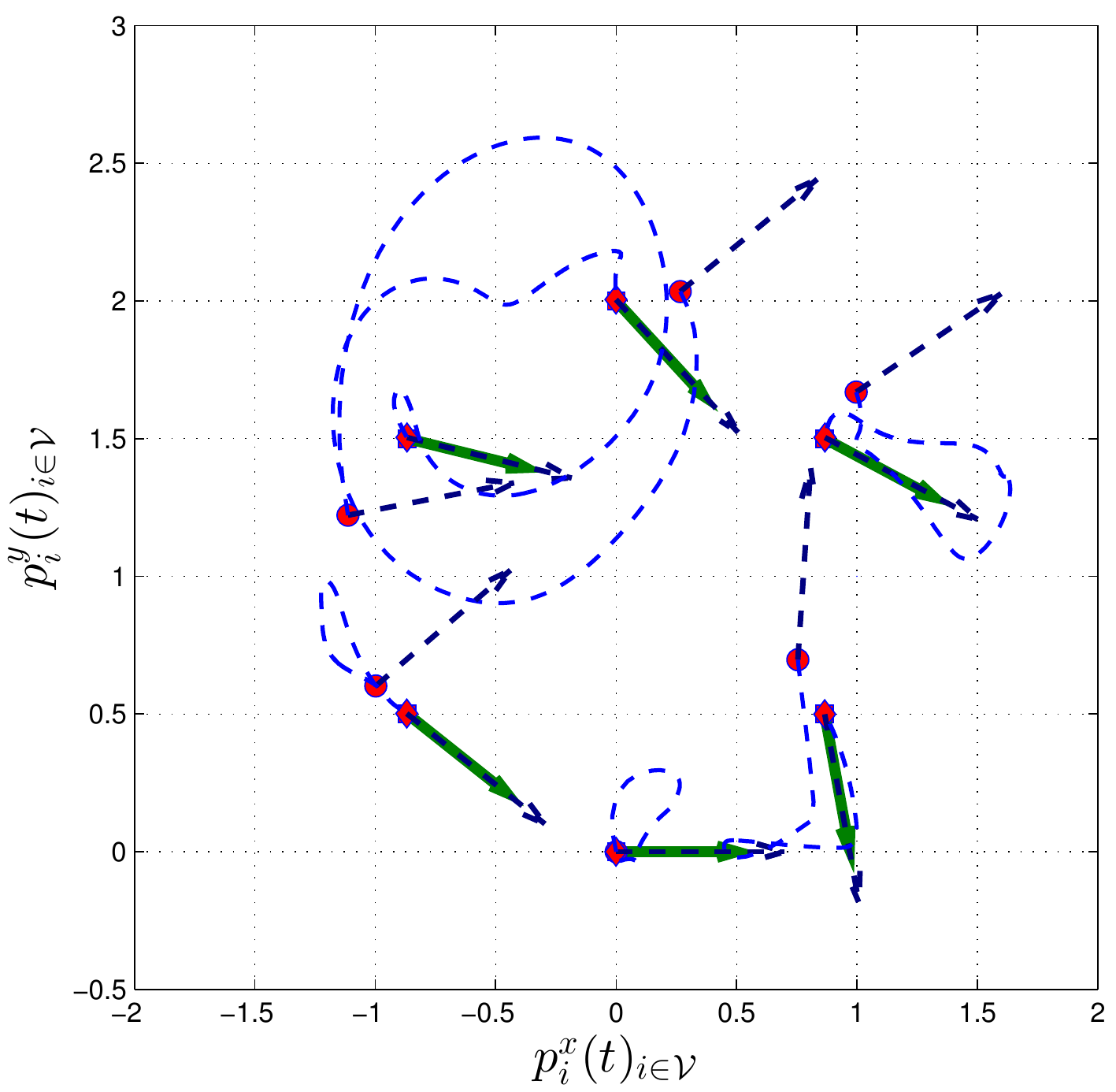}}}
\def\eA{{\includegraphics[width=0.6\columnwidth]{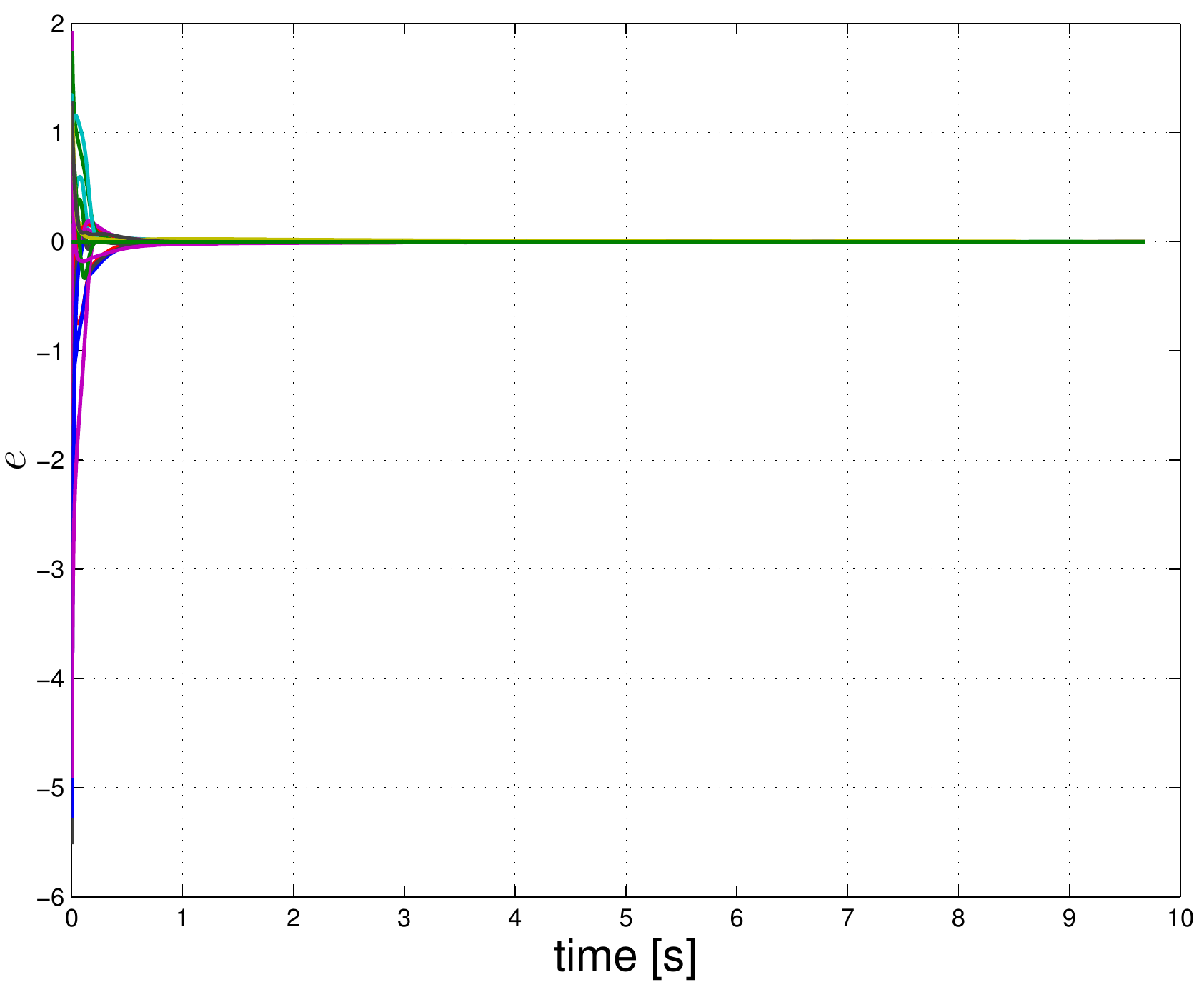}}}
\def\epA{{\includegraphics[width=0.6\columnwidth]{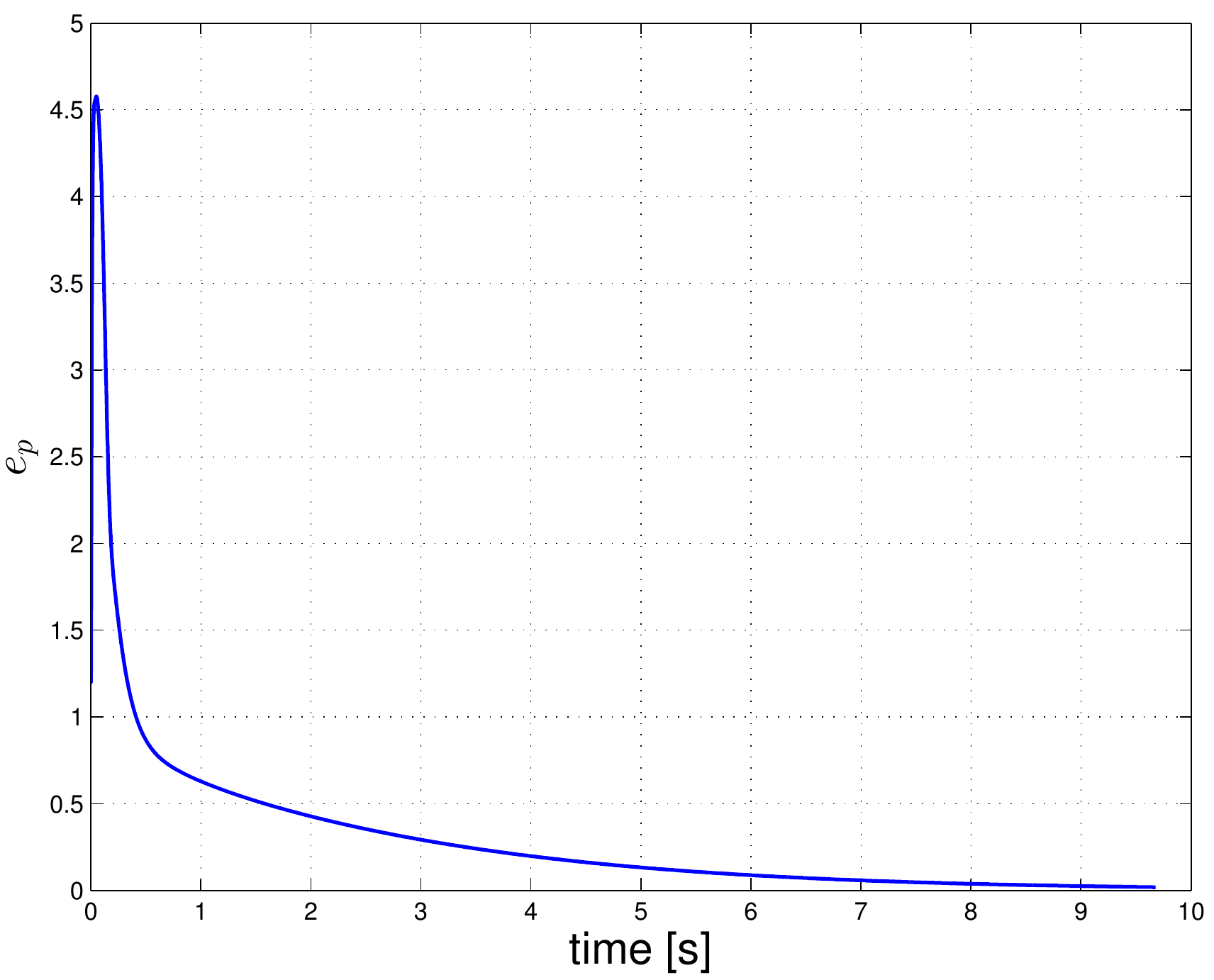}}}

\def\trajB{{\includegraphics[width=0.6\columnwidth]{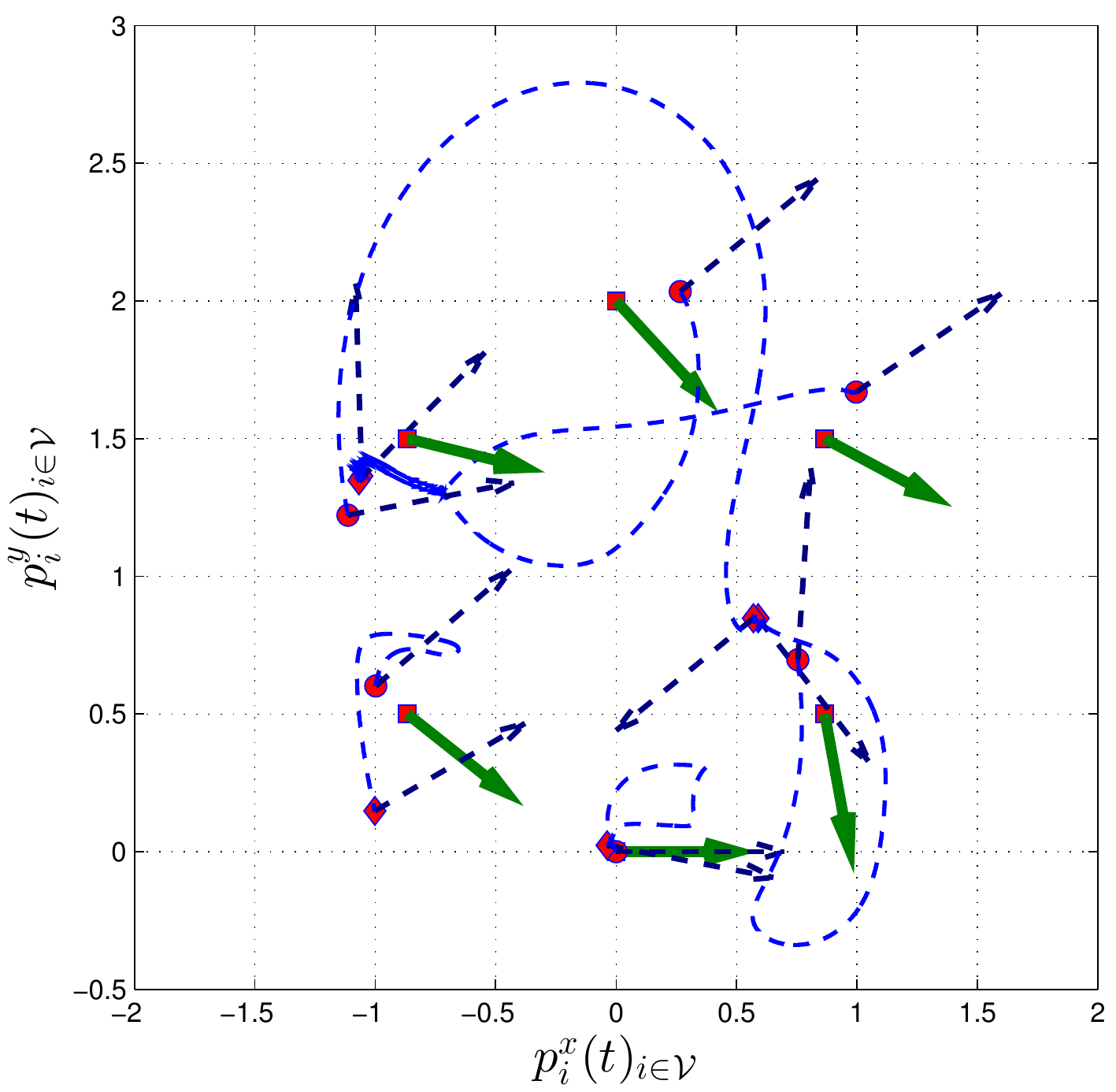}}}
\def\eB{{\includegraphics[width=0.6\columnwidth]{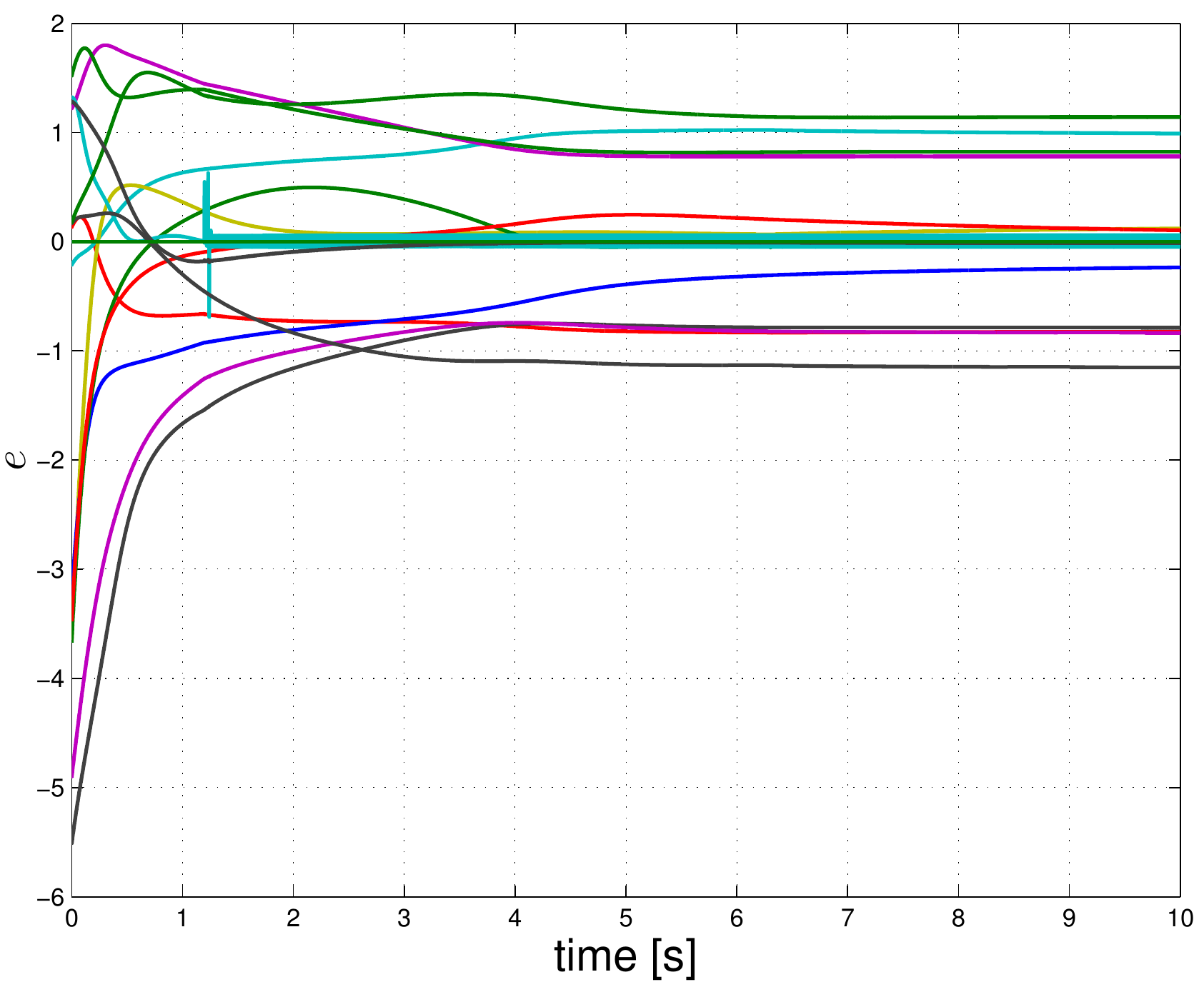}}}
\def\epB{{\includegraphics[width=0.6\columnwidth]{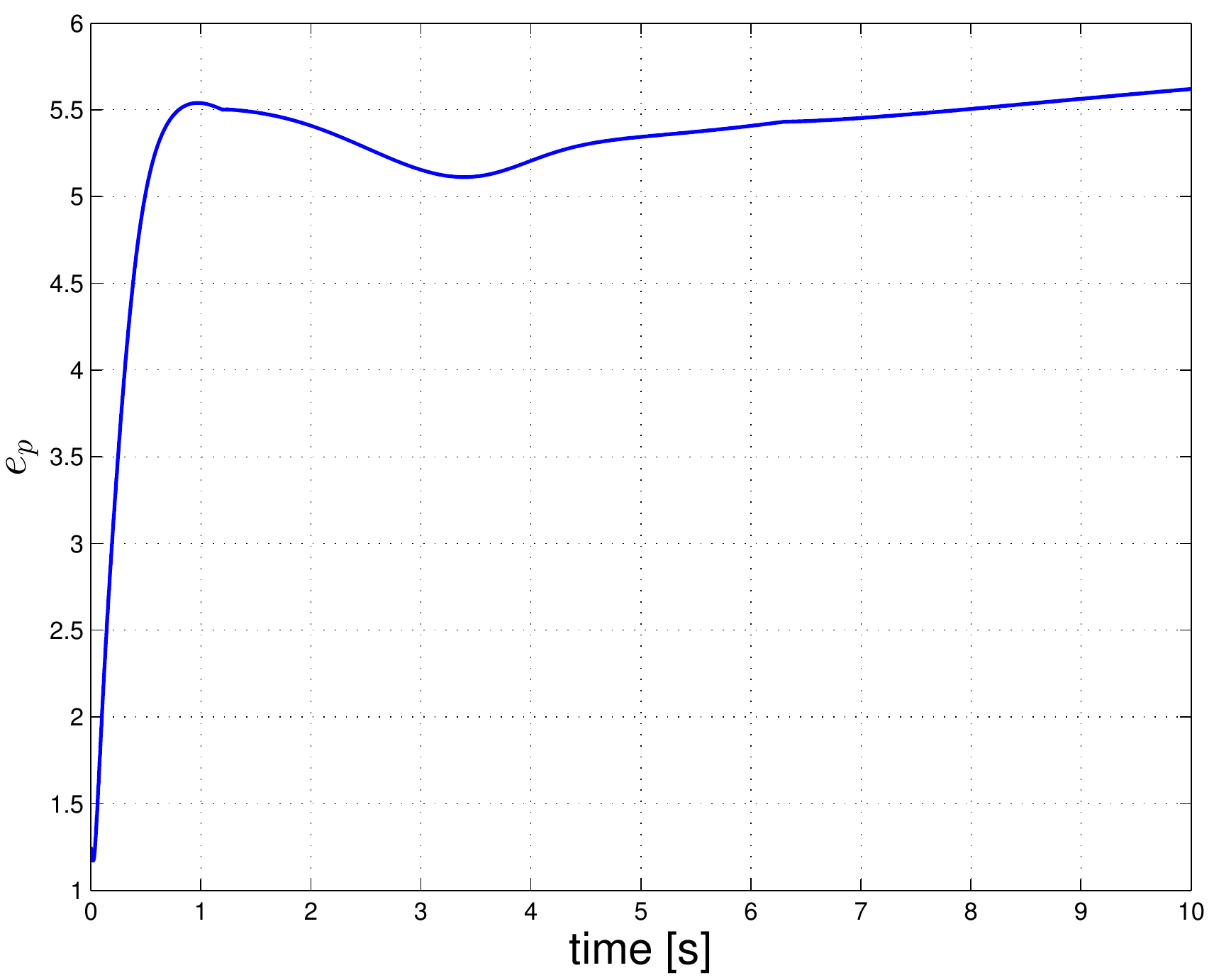}}}

\begin{document}

\title{\LARGE \bf Rigidity Theory in $SE(2)$ for {Unscaled}{} Relative \\Position Estimation using only Bearing Measurements}

\author{Daniel Zelazo, Antonio Franchi, Paolo Robuffo Giordano
\thanks{Daniel Zelazo is with the Faculty of Aerospace Engineering at the Technion-Israel Institute of Technology, Haifa, Israel {\tt \scriptsize dzelazo@technion.ac.il}.} 
\thanks{Antonio Franchi is with with the Max Planck Institute for Biological Cybernetics, Spemannstra\ss{}e 38, 72076 T\"ubingen, Germany {\tt \scriptsize antonio.franchi@tue.mpg.de}.} 
\thanks{Paolo Robuffo Giordano is with CNRS at Irisa and Inria Rennes Bretagne Atlantique, Campus de Beaulieu, 35042 Rennes Cedex, France {\tt \scriptsize prg@irisa.fr}.}
}

\maketitle

\begin{abstract}
This work considers the problem of estimating the{unscaled}{} relative positions of a multi-robot team in a common reference frame from bearing-only measurements.  Each robot has access to a relative bearing measurement taken from the local body frame of the robot, and the robots have no knowledge of a common or inertial reference frame.  A corresponding extension of rigidity theory is made for frameworks embedded in the \emph{special Euclidean group} $SE(2) = \reals^2 \times \mc{S}^1$.  We introduce definitions describing rigidity for $SE(2)$ frameworks and provide necessary and sufficient conditions for when such a framework is \emph{infinitesimally rigid} in $SE(2)$.  Analogous to the rigidity matrix for point formations, we introduce the \emph{directed bearing rigidity matrix} and show that an $SE(2)$ framework is infinitesimally rigid if and only if the rank of this matrix is equal to $2|\mc{V}|-4$, where $|\mc{V}|$ is the number of agents in the ensemble.  The directed bearing rigidity matrix and its properties are then used in the implementation and convergence proof of a distributed estimator to determine the {unscaled}{} relative positions in a common frame.  Some simulation results are also given to support the analysis.
\end{abstract}

\IEEEpeerreviewmaketitle

\section{Introduction}

Control and estimation problems for teams of mobile robots poses many challenges for real-world implementations.   These problems are motivated by diverse application domains including deep space interferometry missions, distributed sensing and data collection, and civilian search and rescue operations, amongst others \cite{2002-AkySanCay, 2008-AndFidYu_Van, Bristow2000, 2012q-FraMasGraRylBueRob,2010-MesEge, 2012f-FraSecRylBueRob, 2006-Mur,2013g-FraOriSte}.  Many of these applications involve operating a robot team in what can be considered as a \emph{harsh environment}.  That is, access to certain measurements in a common reference frame (i.e., inertial position measurements from GPS) are not available.  This motivates control and estimation strategies that can rely on sensing and communication capabilities that do not depend on knowledge of a common reference frame.

When \emph{range measurements} are available then the theory of \emph{formation rigidity} provides the correct framework for considering formation control problems.  Rigidity is a combinatorial theory for characterizing the ``stiffness" or ``flexibility" of structures formed by rigid bodies connected by flexible linkages or hinges.  It has found numerous applications in various engineering sciences and also as a formal mathematical discipline \cite{belabbas_globalstabformation_TAC2013, Connelly2009, Jacobs1997, 1970-Lam, 2009-ShaFidAnd, 1985-TayWhi}.  In \cite{2009-KriBroFra} it was shown that formation stabilization using distance measurements can be achieved only if rigidity of the formation is maintained.  
Formation rigidity also provides a necessary condition for estimating relative positions using only relative distance measurements \cite{2006-AspEreGolMorWhiYanAndBel,2010-CalCarWei}.  Distributed control strategies for dynamically maintaining the rigidity property of a formation was recently considered by the authors in \cite{2012-ZelFraRob, Zelazo2013a}. 

A related concept to formation rigidity is known as \emph{parallel rigidity}.  Whereas rigidity theory is useful for maintaing formations with \emph{fixed distances} between neighboring agents, parallel rigidity focuses on maintain \emph{formation shapes}; that is it attempts to keep the \emph{bearing vector} between neighboring agents constant.   Parallel rigidity was used in \cite{Bishop2011},  \cite{Eren2012}, and \cite{Franchi2012a} for deriving distributed control laws for controlling formations with bearing measurements.  In, \cite{Eren2007}, parallel rigidity was used for the localization problem in robotic networks using bearing measurements. {In~\cite{2012f-FraSecRylBueRob} the authors proposed a bearing-only formation controller for agents in 3-dimensional space requiring only relative bearing measurements, converging almost globally, and maintaining bounded inter-agent distances despite the lack of direct metric information.}{}

The concepts of formation and parallel rigidity have practical relevance for multi-agent systems in that they provide the appropriate analytical framework for defining formations obtained from sensed measurements.  For formation rigidity, the measurements are the form of \emph{distances}, while for parallel rigidity they are \emph{directions}.  In both cases, however, it is assumed that the robots or agents comprising the systems are essentially point-masses; they have no orientation relative to a common world frame.  In many real-world scenarios, however, the sensors used to obtain relative measurements (bearing, distance, etc.) are likely to be physically coupled to the frame of the robot.  Furthermore, the sensors might also introduce additional constraints such as field-of-view restrictions or line-of-sight requirements.  In these scenarios, the attitude of each agent must be considered to define the sensing graph.

In many distributed control strategies for multi-robot teams using relative sensing, an implicit requirement is the team have knowledge of a common reference frame to generate the correct velocity input vectors.  This information is either known directly from special sensors or communication with agents endowed with this information, or it must be estimated by each agent.  This problem was considered in \cite{Franchi2012a} for special classes  of graphs {(and extended to generic graphs using communication)}{} and in \cite{Zelazo2013a} when only distance measurements are available.  

\subsection*{{Related Work and Contribution}{}} \label{prelim}

This paper considers the {unscaled}{} relative  position {(URP)}{} estimation problem for a team of agents  that have access to bearing measurements. 
{The adjective `unscaled' means that the positions of the agents are estimated up to a common scale factor.}{}
 The bearing sensor is attached to the body frame of each agent, and consequently the \emph{attitude} of each agent (as measured from a common inertial frame) will influence which agents can be sensed.  In this direction, we consider each agent as a point in $SE(2)$; it has a position coordinate in $\reals^2$ and an attitude on the 1-dimensional manifold on the unit circle, $\mc{S}^1$.  The bearing measurements available for each agent induces a \emph{directed} sensing graph.  A contribution of this work is to provide necessary and sufficient conditions on the underlying sensing graph and positions of each agent in $SE(2)$ for solving the {URP} {relative position} estimation problem with only bearing measurements.

{Estimation using only relative bearings  as exteroceptive measurements has been considered  also in~\cite{2011k-SteCogFraOri,2012c-CogSteFraOriBue}. However, in those works the robots also had access to egomotion sensors in order to disambiguate the anonymity of the measurements.  This is in contrast to the method proposed here which which does not require such sensors.}{}

  Another similar problem set-up was also considered in \cite{Eren2003, Eren2012, Bishop2011}.  The main distinction with this work is the insistence that the bearing measurements between agents are expressed in the \emph{local} frame of the agent.  
This turns out to be an important assumption and requires a new extension to the theory of rigidity.

This then motivates the study of rigidity for formations in $SE(2)$, which is the main contribution of this work.  Similar to parallel rigidity, the objective for formations in $SE(2)$ is to define a formation shape while also maintaing the relative bearings between each agent.  The main distinction is the bearing measurements are expressed in the \emph{local} frame of each agent, and the corresponding statements on $SE(2)$ rigidity explicitly handle this distinction.  Our approach is to mirror the development of formation rigidity, such as can be found in \cite{Asimow1979}, but for frameworks where each node in the directed graph is mapped to a point in $SE(2)$.  We derive a matrix we term the \emph{directed bearing rigidity matrix} and show that a formation is infinitesimally rigid in $SE(2)$ if and only if the dimension of the kernel of this matrix is equal to four.  Furthermore, we show the infinitesimal motions that span the kernel are the trivial motions of a formation in $SE(2)$, namely the translations, dilations, and coordinated rotations of the formation.  The directed bearing rigidity matrix appears in the relative position estimator and provides the essential ingredient for the convergence proof of the estimator.

The paper is organized as follows.  A brief review of concepts from rigidity theory with an emphasis on parallel rigidity is provided in $\S$\ref{sec:parallel_rigidity}.  The development of rigidity theory for $SE(2)$ is given in $\S$\ref{sec:SE2_rigidity}.  The relative position estimation problem is given in $\S$\ref{sec:estimator}, and some numerical simulation examples are given in $\S$\ref{sec:sims}.  Finally, concluding remarks and future research directions are discussed in $\S$\ref{sec:conclusion}.

\subsection*{Preliminaries and Notations} \label{prelim}
The notation employed is standard.  Matrices are denoted by capital letters (e.g.,~$A$), and vectors by lower case letters (e.g.,~$x$).  The rank of a matrix $A$ is denoted $\rk[A]$.  Diagonal matrices will be written as $D = \diag{\{d_1,\ldots,d_n\}}$.  A matrix and/or a vector that consists of all zero entries will be denoted by ${\bf 0}$; whereas, `$0$' will simply denote the scalar zero.  The $n \times n$ identity matrix is denoted as $I_n$.  The set of real numbers will be denoted as $\reals$, the 1-dimensional manifold on the unit circle as $\mc{S}^1$, and $SE(2) = \reals^2 \times \mc{S}^1$ is the Special Euclidean Group 2. The standard Euclidean $2$-norm for vectors is denoted $\| \, . \,  \|$.  The Kronecker product of two matrices $A$ and $B$ is written as $A \otimes B$ \cite{Horn1991}.  For sets $A$ and $B$, $A-B$ denotes the set difference, $A-B = \{x  \,|\, x \in A, \, x \notin B\}$.  The null-space of an operator $F$ is denoted $\mc{N}\left[F\right]$.

Directed graphs and the matrices associated with them will be widely used in this work; see, e.g., \cite{2001-GodRoy}.  A directed graph $\mc{G}$ is specified by a vertex set $\mc{V}$, an edge set $\mc{E} \subseteq \mc{V} \times \mc{V}$ whose elements characterize the incidence relation between distinct pairs of $\mc{V}$.  A directed edge $e = (v,u) \in \mc{E}$ is an ordered pair, and $v$ is called the head of $e$ and $u$ the tail of $e$.  The \emph{neighborhood} of the vertex $i$ is the set $\mc{N}_i = \{j \in \mc{V} \, | \, (i,j) \in \mc{E}\}$, and the \emph{out-degree} of vertex $i$ is $d_{out}(i)=|\mc{N}_i|$.  
The incidence matrix $E(\mc{G}) \in \reals^{|\mc{V}| \times |\mc{E}|}$ is a $\{0,\pm 1\}$-matrix
with rows and columns indexed by the vertices and edges of $\mc{G}$ such that $[E(\mc{G})]_{ik}$ has the value `$+1$' if node $i$ is the head of edge $e_k$, `$-1$' if it is the tail of $e_k$, and `0' otherwise.  The \emph{complete directed graph}, denoted $K_{|\mc{V}|}$ is a graph with all possible directed edges (i.e. $|\mc{E}| = |\mc{V}|\left(|\mc{V}|-1\right)$).  The graph Laplacian of the matrix $\mc{G}$ is defined as $L(\mc{G}) = E(\mc{G})E(\mc{G})^T$.

\section{Parallel Rigidity Theory}\label{sec:parallel_rigidity}

In this section we briefly review some fundamental concepts of parallel rigidity.  For an overview on distance rigidity theory, please see \cite{Asimow1979, Jackson2007}.  A more detailed treatment parallel rigidity can be found in \cite{Eren2012, 2003-EreWhiMorBelAnd}.%
Parallel rigidity is built upon the notion of a \emph{bar-and-joint framework} consisting of an undirected graph $\mc{G}=(\mc{V},\mc{E})$ and a function mapping each node of the graph to a point in Euclidean space.  In this work we consider the space $\reals^2$ and denote the map as $p : \mc{V} \rightarrow \reals^2$.  Thus, a framework is the pair $(\mc{G},p)$.  In the following we denote by $p(\mc{V}) = \leftm{ccc} p(v_1)^T & \cdots & p(v_{|\mc{V}|}) \rightm \in \reals^{2|\mc{V}|}$ the stacked position vector for the framework.

Parallel rigidity is concerned with angles formed between pairs of points and the lines joining them (i.e. the edges in the graph).  These angles are measured with respect to some common reference frame.  

\begin{definition}[{\small Equivalent Frameworks}]
Two frameworks $(\mc{G},p_1)$ and $(\mc{G},p_2)$ are \emph{equivalent} if $((p_1(v_i)-p_1(v_j))^\perp)^T(p_2(v_i)-p_2(v_j)) = 0$ for all $\{v_i,v_j\} \in \mc{E}$, where $x^\perp$ denotes a $\pi/2$ counterclockwise rotation of the vector $x$.
\end{definition}

\begin{definition}[{\small Congruent Frameworks}]
Two frameworks $(\mc{G},p_1)$ and $(\mc{G},p_2)$ are \emph{congruent} if $((p_1(v_i)-p_1(v_j))^\perp)^T(p_2(v_i)-p_2(v_j)) = 0$ for all pairs $v_i,v_j \in \mc{V}$.
\end{definition}

Observe that for two frameworks to be congruent requires that the line segment between any pair of nodes in one framework is \emph{parallel} to the corresponding segment in the other framework.  Thus, two parallel congruent frameworks are related by an appropriate sequence of \emph{rigid-body translations} and \emph{dilations} of the framework.

\begin{definition}[{\small Global Rigidity}] 
A framework $(\mc{G},p)$ is \emph{parallel globally rigid} if all parallel equivalent frameworks to $(\mc{G},p)$ are also parallel congruent to $(\mc{G},p)$.  
\end{definition}

Consider now a trajectory defined by the time-varying position vector $q(t)\in \reals^{2|\mc{V}|}$.  We consider trajectories that are equivalent to a given framework $(\mc{G},p)$ for all time.  This induces a set of linear constraints that can be expressed as
\bea
 ((p(v_i)-p(v_j))^{\perp})^T(\dot{q}_i(t)-\dot{q}_j(t)) &=0& \label{par_infrigid}
\eea
for all $\{v_i,v_j\} \in \mc{E}$.  Here we employed a short-hand notation $q_i(t)$ to denote the position of node $v_i$ in the time-varying framework $(\mc{G},q(t))$.  The velocities $\dot{q}(t)$ that satisfy the above constraints are referred to as the \emph{infinitesimal motions} of a framework.  
Frameworks with infinitesimal motions that satisfy (\ref{par_infrigid}) and result in only rigid-body translations and dilations are known as \emph{infinitesimally rigid}.

The $|\mc{E}|$ linear constraints given in (\ref{par_infrigid}) can be equivalently written in matrix form as
\bea
R_{\|,\mc{G}}(p(\mc{V}))\dot{q}(t) &=&0 \label{par_infrigid_mtx}.
\eea
The matrix $R_{\|,\mc{G}}(p(\mc{V})) \in \reals^{|\mc{E}|} \times 2|\mc{V}|$ is referred to as the \emph{parallel rigidity matrix}.  The null-space of these matrices thus describe the infinitesimal motions.  The main result of this section is summarized below.

\begin{thm}
A framework $(\mc{G},p)$ is parallel infinitesimally rigid if and only if $\rk[R_{\|,\mc{G}}(p(\mc{V}))] = 2|\mc{V}|-3$.   Furthermore, the three dimensional null-space of the parallel rigidity matrix are correspond to rigid-body translations and dilations of the framework. 
\end{thm}

\section{Rigidity in $SE(2)$}\label{sec:SE2_rigidity}

The concepts of distance and parallel rigidity introduced in $\S$\ref{sec:parallel_rigidity} provides a framework for describing formation shapes in $\reals^2$.  In this section, we extend these notions of rigidity for frameworks that are embedded $SE(2)$.
Our discussion follows closely the presentation of rigidity given in \cite{Asimow1979, Krick2009}.  To begin, we first modify the traditional bar-and-joint framework to handle points in $SE(2)$ as opposed to the Euclidean space $\reals^2$.

\begin{definition}\label{def:SE2-barjoint}
An \emph{$SE(2)$ framework} is the triple  $(\mc{G}, p,\psi)$, where $\mc{G}=(\mc{V},\mc{E})$ is a \emph{directed graph}, $p : \mc{V} \rightarrow \reals^2$ and $\psi : \mc{V} \rightarrow {\mc{S}^1}$ maps each vertex to a point in $SE(2)= \reals^2 \times \mc{S}^1$.
\end{definition}

\AF{
\begin{definition}\label{def:SE2-barjoint}
The $SE(2)$ \emph{bar-and-joint framework} is the couple  $(\mc{G}, q)$, where $\mc{G}=(\mc{V},\mc{E})$ is a \emph{directed graph}, $q : \mc{V} \rightarrow SE(2)$ maps each vertex to a point in $SE(2)$.
\end{definition}
}

We denote by $\chi(v)=(p(v),\psi(v)) \in SE(2)$ the position and attitude vector of node $v \in \mc{V}$. 
For notational convenience, we will refer to the vectors $\chi_p=p(\mc{V}) \in \reals^{2|\mc{V}|}$ and $\chi_\psi=\psi(\mc{V}) \in {\mc{S}^1}^{|\mc{V}|}$ as the position and attitude components of the complete framework configuration.  The vector $\chi(\mc{V}) \in SE(2)^{|\mc{V}|}$ is the stacked position and attitude vector for the complete framework.  We also denote by $\chi_p^x \in \reals^{|\mc{V}|}$ ($\chi_p^y$) as the $x$-coordinate ($y$-coordinate) vector for the framework configuration.

The defining feature of rigidity in $SE(2)$ is the specification of formations that maintain the \emph{relative bearing} angle between points in the framework with respect to the \emph{local frame} of each point.  This is motivated by scenarios where a robot in a multi-robot team is able to measure the relative bearing between itself and other robots.  The explicit use of \emph{directed graphs} in the definition of $SE(2)$ frameworks reinforces this motivation when considering that relative bearing sensors are likely to be attached to the body frame of the robots, and will have certain constraints such as field-of-view restrictions that may exclude certain measurements, and in particular, bidirectional or symmetric measurements. 

In this venue, we assume that a point $\chi(v) \in SE(2)$ has a bearing measurement of the point $\chi(u)$ if and only if the directed edge $(v,u)$ belongs to the graph $\mc{G}$ (i.e., $(v,u) \in \mc{E}$); this measurement is denoted $\beta_{vu} \in {\mc{S}^1}$.  The relative bearing is measured from the body coordinate system of that point.  

We now define the \emph{directed bearing rigidity function} associated with the $SE(2)$ framework, $b_{\mc{G}} : SE(2)^{|\mc{V}|} \rightarrow {\mc{S}^1}^{|\mc{E}|}$, as
\bea \label{bearing_rigidity_fcn}
b_{\mc{G}}(\chi(\mc{V})) = \leftm{ccc} \beta_{e_1} & \cdots & \beta_{e_{|\mc{E}|}} \rightm^T;
\eea
we use the notation $e_i \in \mc{E}$ to represent a directed edge in the graph and assume a labeling of the edges in $\mc{G}$.

The bearing measurement can be equivalently written as a unit vector pointing from the body coordinate of the point $\chi(v)$ to the point $\chi(u)$, i.e., 
\bea \label{unit_bearing_vector}
r_{vu}(p,\psi) = \leftm{c}r_{vu}^{x} \\ r_{vu}^{y}\rightm= \leftm{c} \cos(\beta_{vu})\\ \sin(\beta_{vu})\rightm, 
\eea
which also satisfies the relationship 
$$ \beta_{vu} = \mbox{atan}\left(\frac{r_{vu}^y}{r_{vu}^x}\right).$$
Observe, therefore, that the bearing measurement can be expressed directly in terms of the relative positions and attitudes of the points expressed in the world frame,
\beas 
r_{vu} (p,\psi)&=& \leftm{cc} \cos(\psi(v)) & \sin(\psi(v)) \\ -\sin(\psi(v)) & \cos(\psi(v)) \rightm  \frac{(p(u)-p(v))}{\|p(v)-p(u)\|}\\
& =&\R({\psi}(v))^T\frac{(p(u)-p(v))}{\|p(v)-p(u)\|} =\R({\psi}(v))^T \overline{p}_{vu},
\eeas
where the matrix $\R(\psi(v))$ is a rotation matrix from the world frame to the body frame of agent $v$, and $\overline{p}_{vu}$ is a shorthand notation for describing the normalized relative position vector from $v$ to $u$. {See Figure \ref{fig:SE2} for an illustration.}

\begin{figure}[!t]
\begin{center}
  \includegraphics[width=.80\columnwidth]{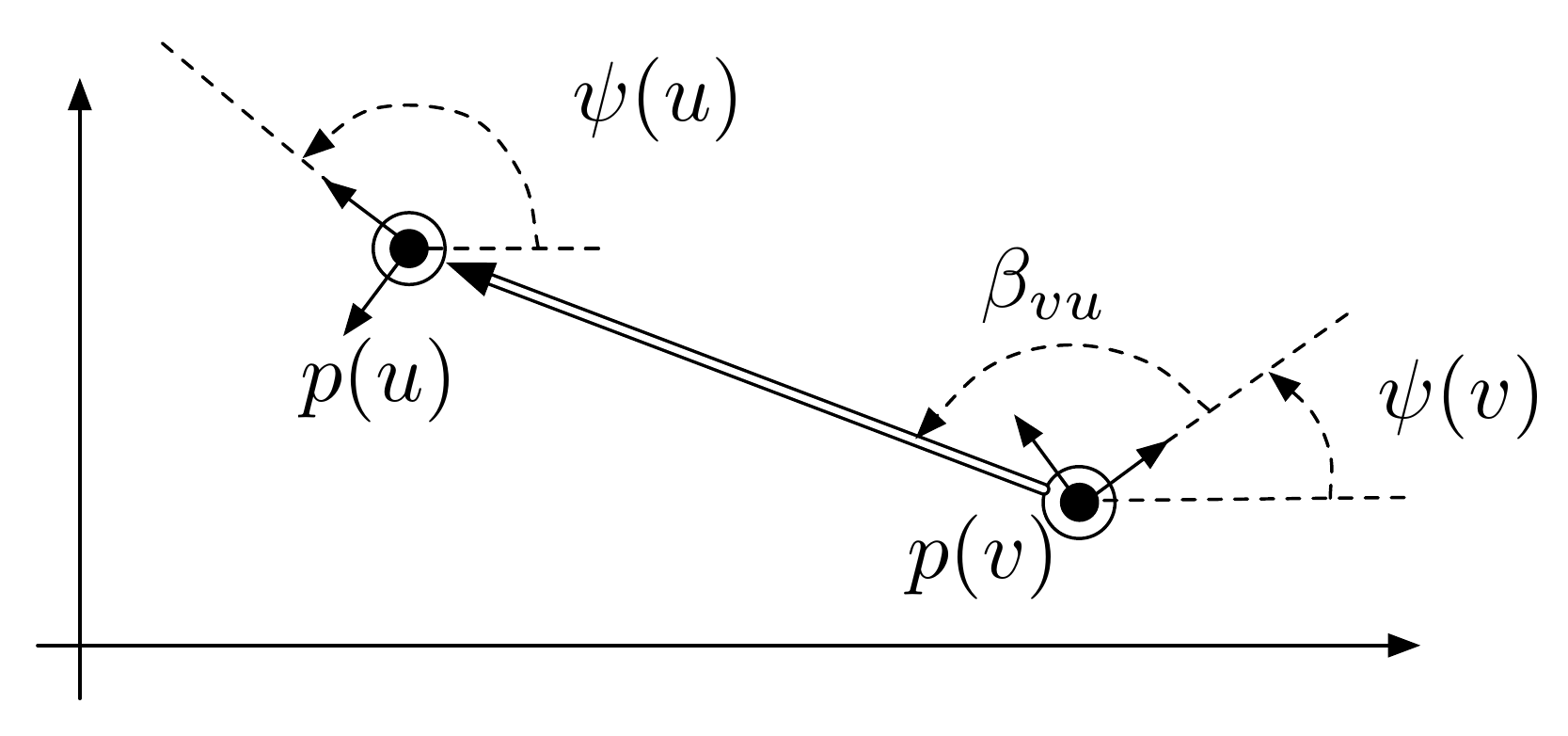}
  \caption{An $SE(2)$ framework with relative bearing measurement from point $\chi(v)$ to point $\chi(u)$.}\label{fig:SE2}
\end{center}
\vspace{-20pt}
\end{figure}

We now introduce formal definitions for rigidity in $SE(2)$, and for the notions of equivalent and congruent formations in $SE(2)$ frameworks.

\begin{definition}[{\small Rigidity in $SE(2)$}]\label{def:se2_rigid}
Let $\mc{G}=(\mc{V},\mc{E})$ be a directed graph and $K_{|\mc{V}|}$ be the complete directed graph on $|\mc{V}|$ nodes.  The $SE(2)$ framework $(\mc{G},p,\psi)$ is \emph{rigid} in $SE(2)$ if there exists a neighborhood $S$ of $\chi(\mc{V}) \in SE(2)^{|\mc{V}|}$ such that 
$$b_{K_{|\mc{V}|}}^{-1}(b_{K_{|\mc{V}|}}(\chi(\mc{V}))) \cap S = b_{\mc{G}}^{-1}(b_{\mc{G}}(\chi(\mc{V}))) \cap S,$$
where $b_{K_{|\mc{V}|}}^{-1}(b_{K_{|\mc{V}|}}(\chi(\mc{V}))) \subset SE(2)$ denotes the pre-image of the point $b_{K_{|\mc{V}|}}(\chi(\mc{V}))$ under the directed bearing rigidity map.

The $SE(2)$ framework $(\mc{G},p,\psi)$ is \emph{roto-flexible} in $SE(2)$ if there exists an analytic path $\eta : [0, \, 1]  \rightarrow SE(2)^{|\mc{V}|}$ such that $\eta(0) = \chi(\mc{V})$ and 
$$\eta(t) \in  b_{\mc{G}}^{-1}(b_{\mc{G}}(\chi(\mc{V}))) - b_{K_{|\mc{V}|}}^{-1}(b_{K_{|\mc{V}|}}(\chi(\mc{V}))) $$
for all $t \in (0, \, 1]$.
\end{definition}

{This definition states that an $SE(2)$ framework $(\mc{G},p,\psi)$ is rigid if and only if for any point $q \in SE(2)$ sufficiently close to $\chi(\mc{V})$ with $b_{\mc{G}}(\chi(\mc{V})) = b_{\mc{G}}(q)$, that there exists a local bearing preserving map of $SE(2)$ taking $\chi(\mc{V})$ to $q$.}  The term \emph{roto-flexible} is used to emphasize that an analytic path in $SE(2)$ can consist of motions in the plane in addition to angular rotations about the body axis of each point.  

 \begin{definition}[{\small Equivalent and Congruent $SE(2)$ Frameworks}]
Frameworks $(\mc{G},p,\psi)$ and $(\mc{G},q,\phi)$ are \emph{bearing equivalent} if 
\bea\label{bearing_eqv} 
\R({\psi}(u))^T\overline{p}_{uv}  =\R({\phi}(u))^T\overline{q}_{uv} ,
\eea
for all $(u,v) \in \mc{E}$ and are \emph{bearing congruent} if 
\beas\label{bearing_cong} 
\R({\psi}(u))^T\overline{p}_{uv} &=& \R({\phi}(u))^T\overline{q}_{uv} \mbox{ and } \\
 \R({\psi}(v))^T\overline{p}_{vu} &=& \R({\phi}(v))^T\overline{q}_{vu},
\eeas
for all $u,v \in \mc{V}$.
\end{definition}

\begin{definition}[{\small Global rigidity of $SE(2)$ Frameworks}]
 A framework $(\mc{G},p,\psi)$ is \emph{globally rigid} in $SE(2)$ if every framework which is bearing equivalent to $(\mc{G},p,\psi)$ is also bearing congruent to $(\mc{G},p,\psi)$.
\end{definition}

It is now worth mentioning a few key distinctions between global rigidity in $SE(2)$ with parallel rigidity in $\reals^2$.  First, parallel rigidity is built on frameworks where the underlying graph is \emph{undirected}.  Rigidity in $SE(2)$, however, is explicitly defined for directed graphs.  
As an example, consider the framework in $SE(2)$ shown in Figures \ref{fig:SE2_rigid1} and \ref{fig:SE2_rigid2}.  Both frameworks are parallel rigid in $\reals^2$ since the internal angles are the same for all agent pairs.  These frameworks, however, are not globally rigid in $SE(2)$.  It can be verified that the two frameworks are equivalent in $SE(2)$ since agent 3 does not actually have any bearing measurements to maintain (the directed graph contains no edges from agent 3 to other agents).  Consequently, agent 3 is free to rotate about its axis without affecting the bearing measurements from the other agents, as shown in Figure \ref{fig:SE2_rigid2}, showing that the frameworks are not congruent.  Observe that adding another directed edge from agent 3 to either agent 1 or 2 will constrain the attitude of agent 3 and the framework will become globally rigid in $SE(2)$.

Motivated by the above example, we now define a corresponding notion of infinitesimal rigidity for $SE(2)$ frameworks.  Using the language introduced in Definition \ref{def:se2_rigid}, we consider a smooth motion along the path $\eta$ with $\eta(0) = \chi(\mc{V})$ such that the initial rate of change of the directed bearing rigidity function is zero.  All such paths satisfying this property are the infinitesimal motions of the $SE(2)$ framework, and are characterized by the null-space of the Jacobian of the directed bearing rigidity function, $\nabla_{\chi} b_{\mc{G}}(\chi(\mc{V}))$, as can be seen by examining the first-order Taylor series expansion of the directed bearing rigidity function,
$$b_{\mc{G}}(\chi(\mc{V}) + \delta \chi) = b_{\mc{G}}(\chi(\mc{V})) + \left(\nabla_{\chi} b_{\mc{G}}(\chi(\mc{V})) \right)\delta \chi + h.o.t.\, , $$
with $\chi(\mc{V}) + \delta \chi$ a point along the path defined by $\eta$.

In this venue, we introduce the \emph{directed bearing rigidity matrix}, $\mc{B}_{\mc{G}}(\chi(\mc{V}))$ as the Jacobian of the directed bearing rigidity function,
\bea\label{directed_bearing_rigidity_matrix}
\mc{B}_{\mc{G}}(\chi(\mc{V})) := \nabla_{\chi} b_{\mc{G}}(\chi(\mc{V})) \in \reals^{|\mc{E}| \times 3 |\mc{V}|}.
\eea
If a path $\eta$ is contained entirely in $ b_{K_{|\mc{V}|}}^{-1}(b_{K_{|\mc{V}|}}(\chi(\mc{V})))$ for all $t \in [0, 1]$, then the infinitesimal motions are entirely described by the tangent space to $b_{K_{|\mc{V}|}}^{-1}(b_{K_{|\mc{V}|}}(\chi(\mc{V})))$, that we denote by $T_p$.  Furthermore, the space $T_p$ must therefore be a subspace of the kernel of the directed bearing rigidity matrix for any other graph $\mc{G}$, i.e. $T_p \subseteq \mc{N} \left[ \mc{B}_{\mc{G}}(\chi(\mc{V})) \right]$; this follows from the definition of roto-flexible frameworks given in Definition \ref{def:se2_rigid}.  This leads us to a formal definition for infinitesimal rigidity of frameworks in $SE(2)$.

\begin{definition}[{\small Infinitesimal Rigidity in $SE(2)$}]\label{def:inf_rigid_se2}
An $SE(2)$ framework $(\mc{G},p,\psi)$ is \emph{infinitesimally rigid} if $  \mc{N} \left[ \mc{B}_{\mc{G}}(\chi(\mc{V})) \right]=\mc{N} \left[ \mc{B}_{K_{|\mc{V}|}}(\chi(\mc{V})) \right]$.  Otherwise, it is \emph{infinitesimally roto-flexible} in $SE(2)$.
\end{definition}

Definition \ref{def:inf_rigid_se2} leads to the main result of this section which relates the infinitesimal rigidity of an $SE(2)$ framework to the rank of the directed bearing rigidity matrix.

\begin{thm}\label{thm:bearing_rigidity}
An $SE(2)$ framework is infinitesimally rigid if and only if 
$$ \rk[\mc{B}_{\mc{G}}(\chi(\mc{V}))] = 3|\mc{V}|-4.$$
\end{thm}

\begin{figure}[!t]
\begin{center}
	\subfigure[Framework $(\mc{G},p,\psi)$ in $SE(2)$.]{\includegraphics[width=.47\columnwidth]{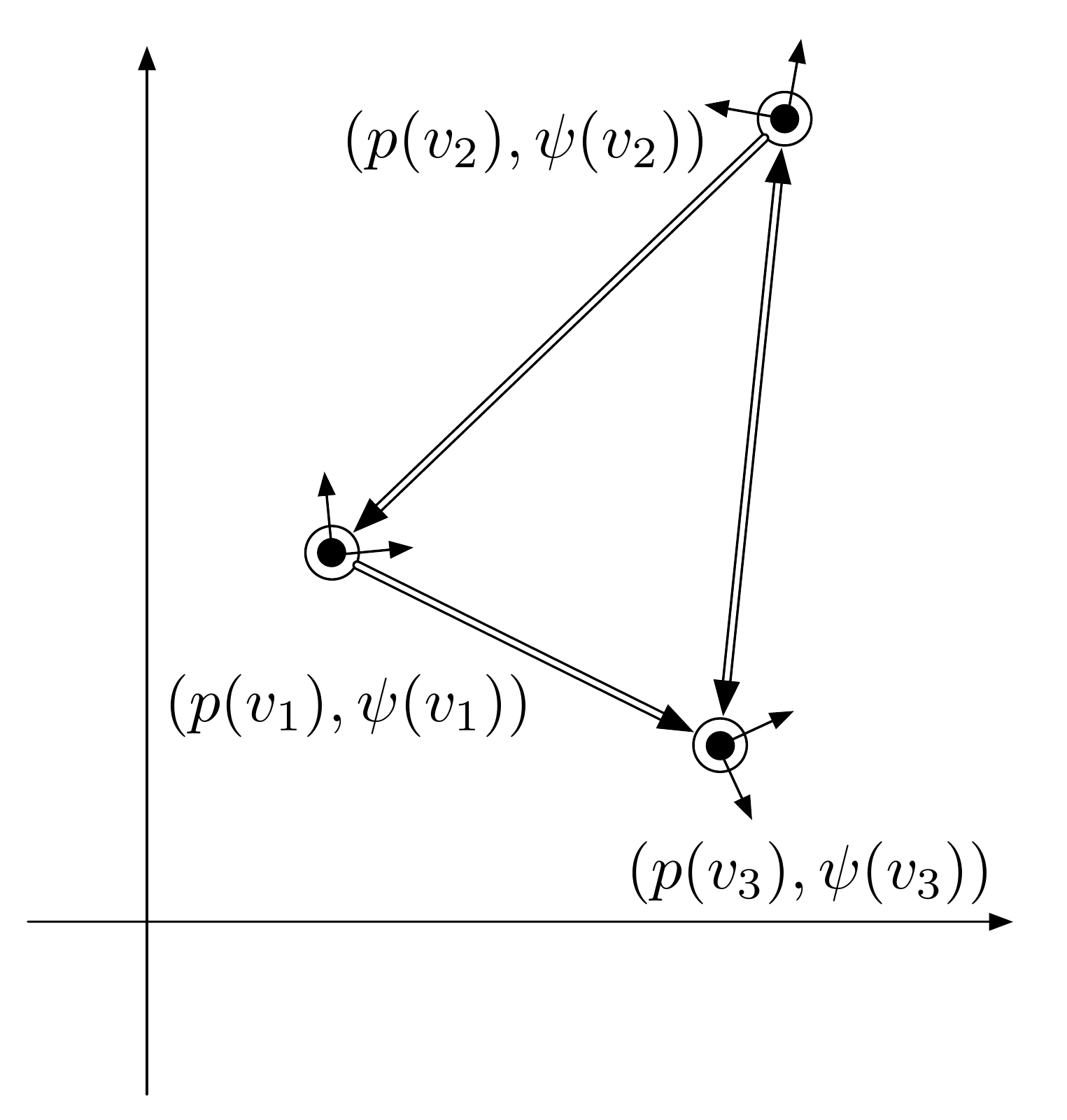}\label{fig:SE2_rigid1}}\hfill
	\subfigure[Framework $(\mc{G},q,\phi)$ in $SE(2)$.]{\includegraphics[width=0.47\columnwidth]{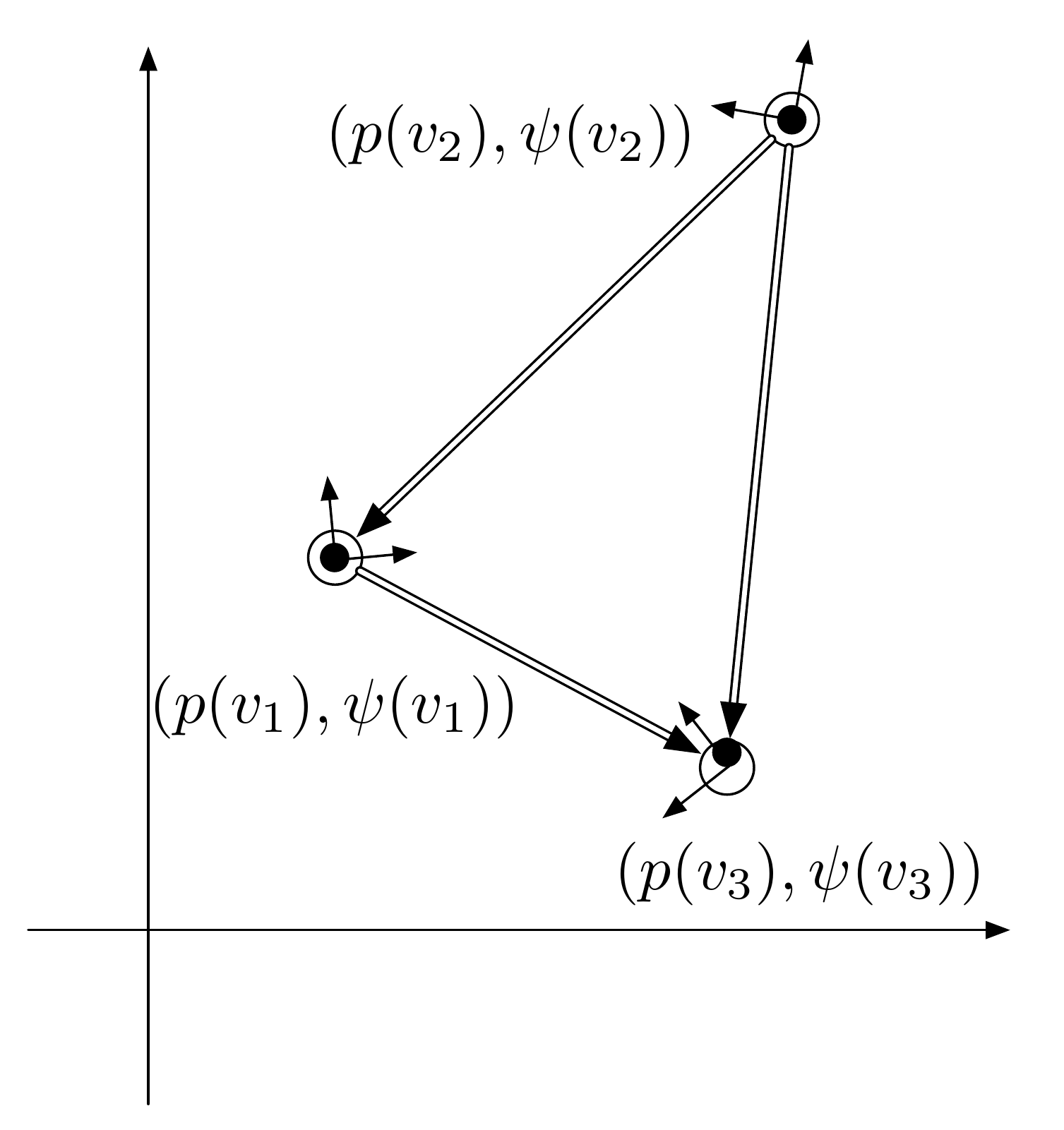}\label{fig:SE2_rigid2}}\hfill
	  \caption{Example of two frameworks that are equivalent but not congruent in $SE(2)$.  If these frameworks are embedded only in $\reals^2$ (i.e. neglecting orientation and as undirected graphs) then they are globally parallel rigid.} \label{fig:rigid_graphs_SE2}
\end{center}
\vspace{-20pt}
\end{figure}

Before proceeding with the proof of Theorem \ref{thm:bearing_rigidity}, we first examine certain structural properties of $\mc{N} \left[ \mc{B}_{\mc{G}}(\chi(\mc{V})) \right]$. First, we observe that the infinitesimal motions of an $SE(2)$ framework are composed of motions in $\reals^2$ with motions in $\mc{S}^1$ for each point.  For an infinitesimal motion $\delta \chi \in   \mc{N} \left[ \mc{B}_{\mc{G}}(\chi(\mc{V})) \right]$, let $\delta \chi_p$ denote the velocity component of $\delta \chi$ in $\reals^{2|\mc{V}|}$ and $\delta \chi_{\psi}$ be the angular velocity component in ${\reals}^{|\mc{V}|}$.

\begin{prop}\label{prop:bearing_cstr_alt}
Every infinitesimal motion $\delta \chi \in \mc{N} \left[ \mc{B}_{\mc{G}}(\chi(\mc{V})) \right]$ satisfies
\bea\label{bearing_rigidity_alt}
R_{\|,\mc{G}}(\chi_{p}) \delta \chi_p = - R_\psi(\chi_{p})  \delta \chi_{\psi}
\eea
\noindent where $R_{\|,\mc{G}}(\chi_{p})$ is the parallel rigidity matrix defined in (\ref{par_infrigid_mtx}) and $R_\psi(\chi_{p})  = D_{\mc{G}}(\chi_{p}) \overline{E}^T(\mc{G})$ with  $D_{\mc{G}}(\chi_{p})  = \diag{\{\ell_{e_1}^2, \cdots, \ell_{e_{|\mc{E}|}}^2\}}$ a diagonal matrix containing the distances squared between all pairs of nodes defined by the edge-set $\mc{E}$, and the matrix $\overline{E} \in \reals^{|\mc{V}| \times |\mc{E}|}$ is defined as 
$$[\overline{E}(\mc{G})]_{ik} = \left\{  \begin{array}{cc} 1, & \mbox{if } e_k=(v_i,v_j) \in \mc{E} \\
								     0, & \mbox{o.w.} \end{array}   \right. .$$
\end{prop}
\begin{proof}
The result in (\ref{bearing_rigidity_alt}) is obtained directly from the evaluation of the Jacobian of the directed bearing rigidity function.
\end{proof}

\begin{rem}
The parallel rigidity matrix as shown in (\ref{bearing_rigidity_alt}) is actually slightly different then what was presented in (\ref{par_infrigid_mtx}).  The main difference is that (\ref{bearing_rigidity_alt}) explicitly considers directed graphs.  Therefore, a bidirectional edge will result in two identical rows in(\ref{bearing_rigidity_alt}), whereas in (\ref{par_infrigid_mtx}) it is treated as a single edge. 
\end{rem}
The first observation from Proposition \ref{prop:bearing_cstr_alt} is the relationship between the infinitesimal motions of an $SE(2)$ framework and those of a parallel rigid framework.  Indeed, if all agents maintain their attitude, i.e. when $\delta \chi_{\psi}=0$, then the constraint reduces to the constraints for parallel rigidity.  The corresponding infinitesimal motions are then the translations and dilations of the framework.  

If the angular velocities of the agents are non-zero, then the infinitesimal motions of the framework correspond to what we term the \emph{coordinated rotations} of the framework.  A coordinated rotation consists of an angular rotation of each agent about its own body axis with a rigid-body rotation of the framework in $\reals^2$.
The coordinated rotations that satisfy (\ref{bearing_rigidity_alt}) are thus related to the subspace
$$ \mathcal{R}_{{\scriptsize \circlearrowright}}( \mc{G})=\mbox{IM}\left\{ R_{\|,\mc{G}}(\chi_p) \right\} \cap \mbox{IM} \left\{-R_{\psi} (\chi_p) \right\} \subset \reals^{|\mc{E}|},$$
that we term the \emph{coordinated rotation subspace}.
Formally, the coordinated rotations can be constructed as
$$\hat{\delta \chi}_p \in R_{\|,\mc{G}}^{-1}[\mathcal{R}_{{\scriptsize \circlearrowright}}( \mc{G})], \mbox{and } \hat{\delta \chi}_{\psi} = -R_{\psi}^{\dagger}(\chi_p)R_{\|,\mc{G}}(\chi_p)\hat{\delta \chi}_p,$$
where by $(A)^{-1}[W]$ we mean the pre-image of the set $W$ under the mapping $A$, and $M^{\dagger}$ is the left-generalized inverse of the matrix $M$.\footnote{ That is, $M^{\dagger}$ satisfies $MM^{\dagger}M=M$.  If $M$ has full rank, then $M^{\dagger}$ is the pseudo-inverse of $M$.}

\begin{prop}\label{prop:cordrot_notempty}
The coordinated rotation subspace is non-trivial.  Equivalently, $\dim\mathcal{R}_{{\scriptsize \circlearrowright}}( \mc{G}) \geq 1$.
\end{prop}
\begin{proof}
We prove this by explicitly constructing a vector in the coordinated rotation subspace.  Consider a rigid-body rotation of the framework in $\reals^2$ described by 
$$z_p = \left(I_{|\mc{V}|} \otimes \leftm{cc}0 & 1 \\ -1 & 0 \rightm \right) \chi_p.$$
It is a straight-forward (although tedious) exercise to verify that $R_{\|,\mc{G}}(\chi_p)z_p = D_{\mc{G}}(\chi_p)\ones_{|\mc{E}|} $.  Furthermore, from the construction of $\overline{E}$ it follows that $\overline{E}^T\ones_{|\mc{V}|}=\ones_{|\mc{E}(\mc{G})|}$ and therefore $R_{\|,\mc{G}}(\chi_p)z_p = D_{\mc{G}}(\chi_p)\overline{E}^T(\mc{G})\ones_{|\mc{V}|}$ concluding the proof.
\end{proof}
The proof of Proposition \ref{prop:cordrot_notempty} formally describes how a coordinated rotation can be constructed for any $SE(2)$ framework.  Each point in the framework should rotate about its own axis at the same rate as the rigid-body rotation of the formation.  This can be considered the $SE(2)$ extension of the infinitesimal motions associated with distance rigidity.  Proposition \ref{prop:cordrot_notempty} can now be used to make a stronger statement about the coordinated rotation subspace for the complete graph.

\begin{prop}\label{prop:complete_coordrot_subspace}
For the complete directed graph $K_{|\mc{V}|}$,  $\dim  \mathcal{R}_{{\scriptsize \circlearrowright}}(K_{|\mc{V}|}) = 1$.
\end{prop}
\begin{proof}
The proof of Proposition \ref{prop:cordrot_notempty} constructs one vector in the coordinated rotation subspace.  Assume that $\dim  \mathcal{R}_{{\scriptsize \circlearrowright}}(K_{|\mc{V}|})  > 1$.  Then there must exist at least one other coordinated rotation that is orthogonal to the one constructed in Proposition \ref{prop:cordrot_notempty} and contains a non-trivial angular rotation of points in the framework. Note that in Proposition \ref{prop:cordrot_notempty} each agent was assigned a unit angular velocity in the same (counter-clockwise) direction.  Thus, any other choice for angular velocities must either be described by each point rotating  in the same direction, but non-uniform velocities, or at least two points rotating in opposite directions.

Considering this observation, it is sufficient to see if such a motion can be constructed for the graph $K_2$.  In this situation, $\overline{E}(K_2) = I_2$ and one can directly conclude from (\ref{bearing_rigidity_alt}) that there can be no additional coordinated rotation then the one described.
\end{proof}

\begin{cor}\label{cor:inf_rigid_SE2}
An $SE(2)$ framework is infinitesimally rigid in $SE(2)$ if and only if 
\begin{enumerate}
\item $\rk[R_{\|,\mc{G}}(\chi_p)] = 2|\mc{V}| - 3$ and
\item $\dim\{\mc{R}_{\circlearrowright}(\mc{G})\} = 1$.
\end{enumerate}
\end{cor}

\begin{proof}
The statement follows directly from Definition \ref{def:inf_rigid_se2}, Proposition \ref{prop:bearing_cstr_alt} and Proposition \ref{prop:cordrot_notempty}.
\end{proof}

We are now ready to prove Theorem \ref{thm:bearing_rigidity}.

\begin{proof}[Proof of Theorem \ref{thm:bearing_rigidity}]
Assume that $ \rk[\mc{B}_{\mc{G}}(\chi(\mc{V}))] = 3|\mc{V}|-4$.  From Propositions \ref{prop:bearing_cstr_alt} and  \ref{prop:complete_coordrot_subspace} we conclude that $ \rk[\mc{B}_{K_{|\mc{V}|}}(\chi(\mc{V}))] = 3|\mc{V}|-4$.  By definition \ref{def:inf_rigid_se2}, we conclude that the $SE(2)$ framework $(\mc{G},p,\psi)$ is infinitesimally rigid.  Assume now that the $SE(2)$ framework is infinitesimally rigid.  By corollary \ref{cor:inf_rigid_SE2}, we conclude $\rk[R_{\|,\mc{G}}(p(\mc{V})] = 2|\mc{V}| - 3$ and $\dim\{\mc{R}_{\circlearrowright \,\mc{G}}\} = 1$.  Therefore, $ \rk[\mc{B}_{\mc{G}}(\chi(\mc{V}))] = 3|\mc{V}|-4$.
\end{proof}

While the general structure of the coordinated rotation subspace can be difficult to characterize for arbitrary graphs, it does lead to a necessary condition on the underlying graph of the framework for infinitesimal rigidity.

\begin{prop}
If an $SE(2)$ framework is infinitesimally rigid, then $d_{out}(v) \geq 1$ for all $v \in \mc{V}$.
\end{prop}
\begin{proof}
Assume that there exists a node $v \in \mc{V}$ such that $d_{out}(v) = 0$.  Then a solution to (\ref{bearing_rigidity_alt}) is $\delta \chi_p = {\bf 0}$ and $[\delta \chi_{\psi}]_i = 1$ if $i$ corresponds to node $v$ and 0 otherwise.  This motion does not belong to the subspace $T_p$ and therefore $\rk[\mc{B}_{\mc{G}}(\chi(\mc{V}))] >3|\mc{V}|-4$ and the framework is not infinitesimally rigid.
\end{proof}

\section{Estimation of Relative Positions}\label{sec:estimator}

Achieving high-level objectives such as formations for multi-robot systems require that all robots have knowledge of a common reference frame.  This is to ensure that their velocity inputs vectors are all consistent when maneuvering to achieve the common formation task.  However, often the sensed data that is available, such as a relative bearing measurement, is measured from the local body frame of each agent.  Furthermore, agents do not have access to a global coordinate system.  A requirement for multi-robot systems, therefore, is the ability to \emph{estimate} a common reference frame in order to express to relative position information.  This section describes how the results from $\S$\ref{sec:SE2_rigidity} can be used to distributedley estimate a common reference frame from only the relative bearing measurements.

In this direction, we consider an infinitesimally rigid $SE(2)$ framework $(\calG,p,\psi)$.  We assume that there are two points in the framework whose Euclidean distance \AFmod{is unknown but positive and constant}{is a known, but arbitrary constant (without loss of generality assumed to be 1)}; these points are indexed as $\iota$ and $\kappa$ (i.e., the position of agent $\iota$ is $p(\iota)$).
Denote with $\hat{\xi}_{\iota i}\in \reals^2$  the estimate of the quantity
\bea\label{eq:scale_free_pos}
\xi_{\iota i}=\R(\psi(\iota))^T \frac{p(i)-p(\iota)}{\|p(\iota)-p(\kappa)\|}
\eea
i.e., the relative position (expressed in the body frame of agent $\iota$) of a virtual point that is on the line connecting agent $\iota$ and a generic agent $i$ and whose distance from $\iota$ is $\frac{\|p(i)-p(\iota)\|}{\|p(\iota)-p(\kappa)\|}$. 

Denote then with $\hat\vartheta_i\in {\mc{S}^1}$ the estimate of the angle $\vartheta(i)$ defined by 
\bea\label{angle_relation}
\R(\vartheta(i))=\R(\psi(i))^T\R(\psi(\iota)),\label{eq:thetadef}
\eea
whose role will be clear in the following.
Define then the following quantities: 
\begin{equation}\label{estimates}
\hat{\xi}_{ij}=\hat{\xi}_{\iota j}-\hat{\xi}_{\iota i},\;
\; \hat r_{ij}=\R(\hat\vartheta_i) \frac{\hat{\xi}_{ij}}{\|\hat{\xi}_{ij}\|},\;
\hat\beta_{ij} = {\rm atan2}(\hat r_{ij}^y,\hat r_{ij}^x).
\end{equation}
Thus the quantity $\hat{\xi}_{ij}$ is an estimate of the relative position vector from $i$ to $j$, scaled by the quantity $\|p(\iota)-p(\kappa)\|$, and expressed in a common reference frame whose origin is $p(\iota)$ and orientation is $\psi(\iota)$. \AFmod{Notice that $\hat{\xi}_{ij}$ represents an unscaled estimate (in the sense explained in the Introduction) of the actual relative position between the agents.}{} Similarly, the estimate of the attitude of the point $i$ can be obtained from (\ref{angle_relation}).

The important fact is that if $\hat\vartheta(i)=\vartheta(i)$ and $\hat{\xi}_{\iota i}$ is equal to~\eqref{eq:scale_free_pos} we  obtain (using also~\eqref{eq:thetadef}) that
\beas
 \hat r_{ij} &=&  \R(\vartheta(i)) \R(\psi(\iota))^T \frac{p(i)-p(j)}{\|p(i)-p(j)\|} \\
 &=& \R(\psi(i))^T \frac{p(i)-p(j)}{\|p(i)-p(j)\|} =  r_{ij} 
\eeas
this justifies the fact that $ \hat r_{ij}$ and $\hat\beta_{ij}$ represent our estimates of $r_{ij}(p,\psi)$, and $\beta_{ij}$, respectively, as defined in (\ref{unit_bearing_vector}).

Our goal can be then recast as the design of an estimator that is able to compute $\hat{\xi}_{\iota i}$ and $\hat\vartheta(i)$ for all $i=1\ldots |\calV|$ using the bearing measurements that corresponds to each directed edge of $\calE$. In order to do so we consider the following estimation error: 
\begin{equation}\label{eq:error_bearing}
e(\hat{\xi},\hat\vartheta,p,\psi) = b_{\mc{G}}(\chi(\mc{V})) - \hat{b}_{\mc{G}}(\hat{\xi},\hat{\vartheta})
\end{equation}
where $\hat{b}_{\mc{G}}(\hat{\xi},\hat{\vartheta}) \in \reals^{|\mc{E}|}$ is the vector of estimated relative bearings obtained from (\ref{estimates}).

The objective of the estimation algorithm can be then stated as the minimization of the following scalar function
\bea\label{eq:est_objective}
J(e) &\hspace{-7pt}=&\hspace{-7pt} \frac{1}{2}\left(k_e\|e(\hat{\xi},\hat\vartheta,p,\psi)\|^2 + k_1\|\hat{\xi}_{\iota \iota}\|^2 + k_2(\|\hat{\xi}_{\iota \kappa}\|^2-1)^2 + \right. \nonumber \\
&&\left. k_3(1-\cos\hat\vartheta(\iota))\right),
\eea
where the nonnegative terms $k_1\|\hat{\xi}_{\iota \iota}\|^2$, $k_2(\|\hat{\xi}_{\iota \kappa}\|^2-1)^2$ and $k_3 (1-\cos\vartheta(\iota))$ account for the fact that at steady state the estimator should let $\hat{\xi}_{\iota \iota}$ converge to $0$, $\|\hat{\xi}_{\iota \kappa}\|$ converge to $1$, and $\hat\vartheta(\iota)$ converge to $0$. The positive gains $k_e,k_1,k_2$, and $k_3$ are introduced here to tune the priority of the single error components within the overall error. 

Minimization of~\eqref{eq:est_objective} can be achieved by following the antigradient of $J(e)$, i.e., by choosing:
\bea\label{eq:estimator}
\begin{pmatrix}
\dot {\hat{\xi}} \\
\dot{\hat\vartheta}
\end{pmatrix}
=
-k_e\left(\nabla_{(\hat{\xi},\hat\vartheta)} e\right)^Te
-
\begin{pmatrix}
\vdots\\
k_1 \hat{\xi}_{\iota \iota}\\
\vdots\\
k_2(\hat{\xi}_{\iota \kappa}^T\hat{\xi}_{\iota \kappa} - 1)\hat{\xi}_{\iota \kappa}\\
\vdots\\
k_3 \sin\hat\vartheta(\iota)\\
\vdots\\
\end{pmatrix}
\eea
where the terms $k_1\hat{\xi}_{\iota \iota}$, $k_2(\hat{\xi}_{\iota \kappa}^T\hat{\xi}_{\iota \kappa} - 1)\hat{\xi}_{\iota \kappa}$, and $k_3\sin\hat\vartheta(\iota) $ appear at the $\iota$-th and $\kappa$-th entry pairs of $\dot{\hat{\xi}}$ and $\iota$-th entry of $\dot{\hat\vartheta}$, respectively, and all the other terms are zero.

As a matter of fact, considering that $b_{\mc{G}}(\chi(\mc{V}))$ is constant, the Jacobian of $e(\hat{\xi},\hat\vartheta,p,\psi)$ can be expressed in terms of the directed bearing rigidity matrix as
\bea\label{grad_error}
\nabla_{(\hat{\xi},\hat\vartheta)} e = -\leftm{cc} D_{\mc{G}}^{-1}(\hat{\xi})R_{\|,\mc{G}}(\hat{\xi}) & \overline{E}(\mc{G})^T\rightm .
\eea
Note that the form above is consistent with (\ref{bearing_rigidity_alt}), which can be obtained from the directed bearing rigidity matrix using an appropriate permutation matrix.

\begin{prop}
If the framework  $(\calG,p,\psi)$ is (infinitesimally) rigid in $SE(2)$ then the vector of true values 
{\tiny
$$\left[ \R(\psi(\iota)) \frac{p(1)-p(\iota)}{\|p(\iota)-p(\kappa)\|}^T \, \cdots \, \R(\psi(\iota)) \frac{p({|\calV|})-p(\iota)}{\|p(\iota)-p(\kappa)\|}^T \, \vartheta(1)\ldots\vartheta({|\calV|})\right]^T$$}
is an isolated local minimizer of $e$. Therefore, there exists an $\epsilon > 0$ such that, for all initial conditions $(\hat{\xi}_0^T, \, \hat\vartheta_0)^T$ whose distance from the true values is less than $\epsilon$, the estimation $\hat{\xi}$ and $\hat\vartheta$ converge  to the true values. 
\end{prop}

\begin{proof}
If the framework is infinitesimally rigid in $SE(2)$, then in any sufficiently small neighborhood of the true bearing values, the only configurations that result $\|e(\hat{\xi},\hat\vartheta,p,\psi)\|^2$ being zero in~\eqref{eq:est_objective} are the trivial motions of the true values (i.e. the rigid-body translations, dilations, and coordinated rotations). For the true values the remaining terms of~\eqref{eq:est_objective} are zero and therefore is $J(e)=0$. If any non-zero trivial motion is applied to the true values then at least one of the remaining terms in $J(e)$ becomes positive. This means that the true values is an isolated local minimizer of~\eqref{eq:est_objective} and that the $J(e)$ is locally convex around the true values. Therefore gradient descent is enough to converge to the true values if the initial error is sufficiently small.   
\end{proof}

\section{Simulation Example}\label{sec:sims}

In this section we report two simulation case studies meant to illustrate the relative position estimator of Sect.~\ref{sec:estimator}. Both simulations involved a total of $|\calV|=6$ agents; the directed sensing graphs are shown in Figs.~\ref{fig:sims_plots}(a,e). By a proper choice of the initial conditions $p(t_0),\,\psi(t_0)$, this purposely resulted in an infinitesimally rigid framework $(\calG_1,\,p(t_0),\,\psi(t_0))$ and a roto-flexible framework $(\calG_2,\,p(t_0),\,\psi(t_0))$.  The following gains were employed: $k_e=5$, $k_1=k_2=k_3=100$. The initial conditions $\hat\xi(t_0)$ and $\hat\vartheta(t_0)$ for the estimator~(\ref{eq:estimator}) were taken as their real values plus a (small enough) random perturbation. 

Figures~\ref{fig:sims_plots}(b--d,f--h) reports the results for the two cases, with the plots in top row (Figs.~\ref{fig:sims_plots}(b--d)) corresponding to the infinitesimal setup, and the plots in the bottom row (Figs.~\ref{fig:sims_plots}(f--h)) to the roto-flexible setup. Let us first consider case~I: Fig.~\ref{fig:sims_plots}(b) shows the behavior of $e(t)$, the error vector between the measured and estimated bearing angles as defined in~(\ref{eq:error_bearing}). We note that under the action of the estimator~(\ref{eq:estimator}), all the $|\calE|$ components of $e(t)$ converge to zero as expected owing to the infinitesimal rigidity of the considered framework. Next, Fig.~\ref{fig:sims_plots}(c) reports the behavior of $e_p(t)=\sum^{|\calV|}_{i=1}\|\xi_{\iota i}-\hat \xi_{\iota i}(t)\|$, i.e., the cumulative error in estimating the unscaled positions $\xi_{\iota i}$ (as defined in~(\ref{eq:scale_free_pos})) for all the $|\calV|$ agents. As expected, $e_p(t)$ converges to $0$ as well (demonstrating again the \emph{rigidity} of the framework). Finally, Fig.~\ref{fig:sims_plots}(d) shows the trajectories of $\hat \xi_{\iota i}(t)$ and $\hat \psi_i(t)$ on the plane (with $\hat \psi_i(t)$ obtained from~(\ref{angle_relation}) when evaluated upon the estimated $\hat\vartheta_i$): here, the real (and constant) poses $(p,\,\psi)$ are indicated by square symbols and thick green arrows, while the initial $\hat \xi_{\iota i}(t_0)$ and $\hat \psi_i(t_0)$ are represented by small circles and dashed black arrows. We can thus note how the estimated position and orientation of every agent converges towards its real value. These results are of course very different for case~II as clear from Figs.~\ref{fig:sims_plots}(f--h) because of the non-rigidity of the employed framework in this case.

\begin{figure*}[!h]
\begin{center}
	\subfigure[The directed graph used in the first simulation associated with an $SE(2)$ infinitesimally rigid framework.]{\quad\quad\quad\quad\includegraphics[width=.5\columnwidth]{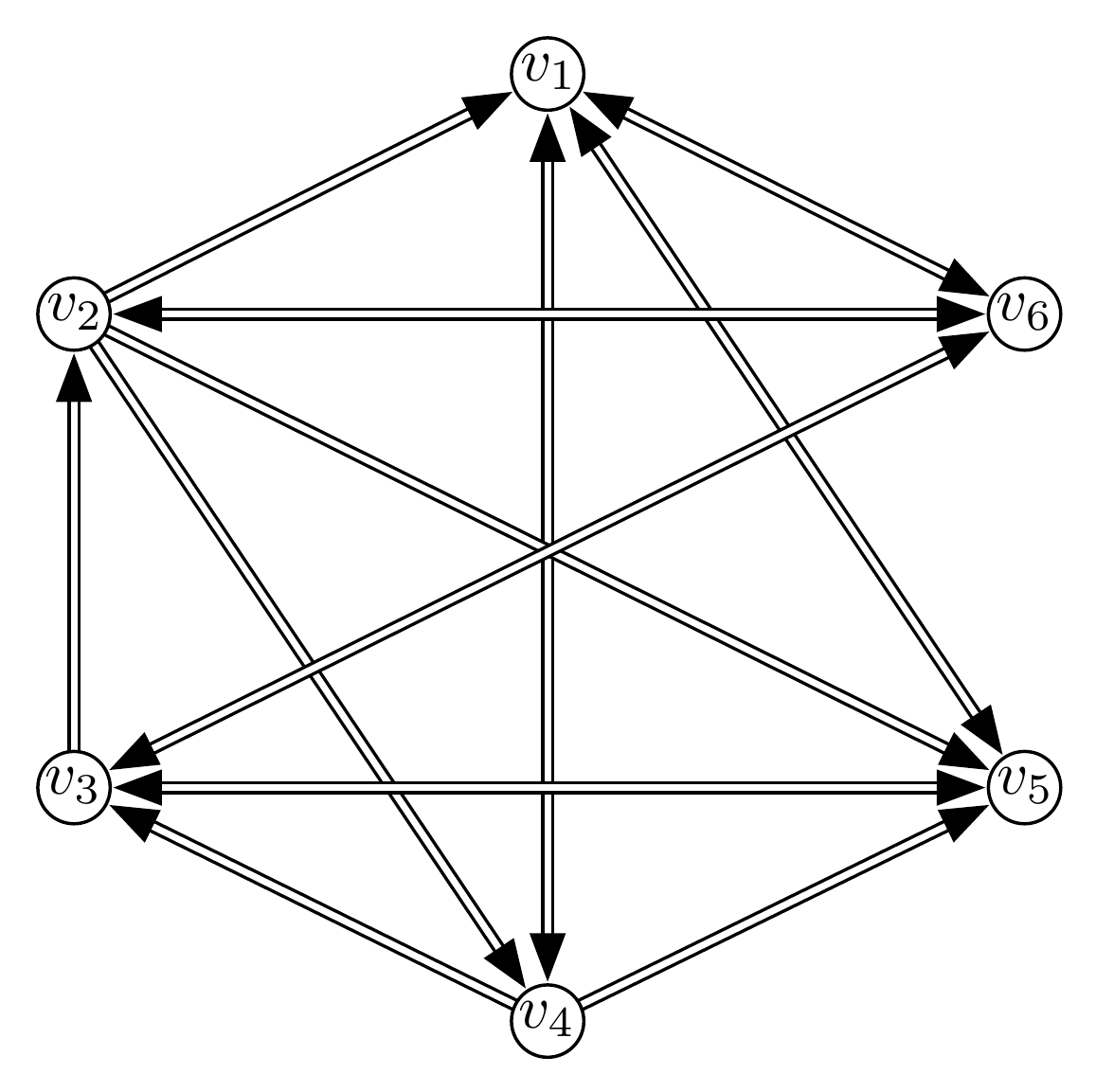}\label{fig:sims1}\quad\quad\quad\quad}\; 
	\subfigure[]\eA \; \subfigure[]\epA  \; \subfigure[]\trajA\\
\subfigure[The directed graph used in the second simulation associated with an $SE(2)$ infinitesimally roto-flexible framework.]{\quad\quad\quad\quad\includegraphics[width=0.5\columnwidth]{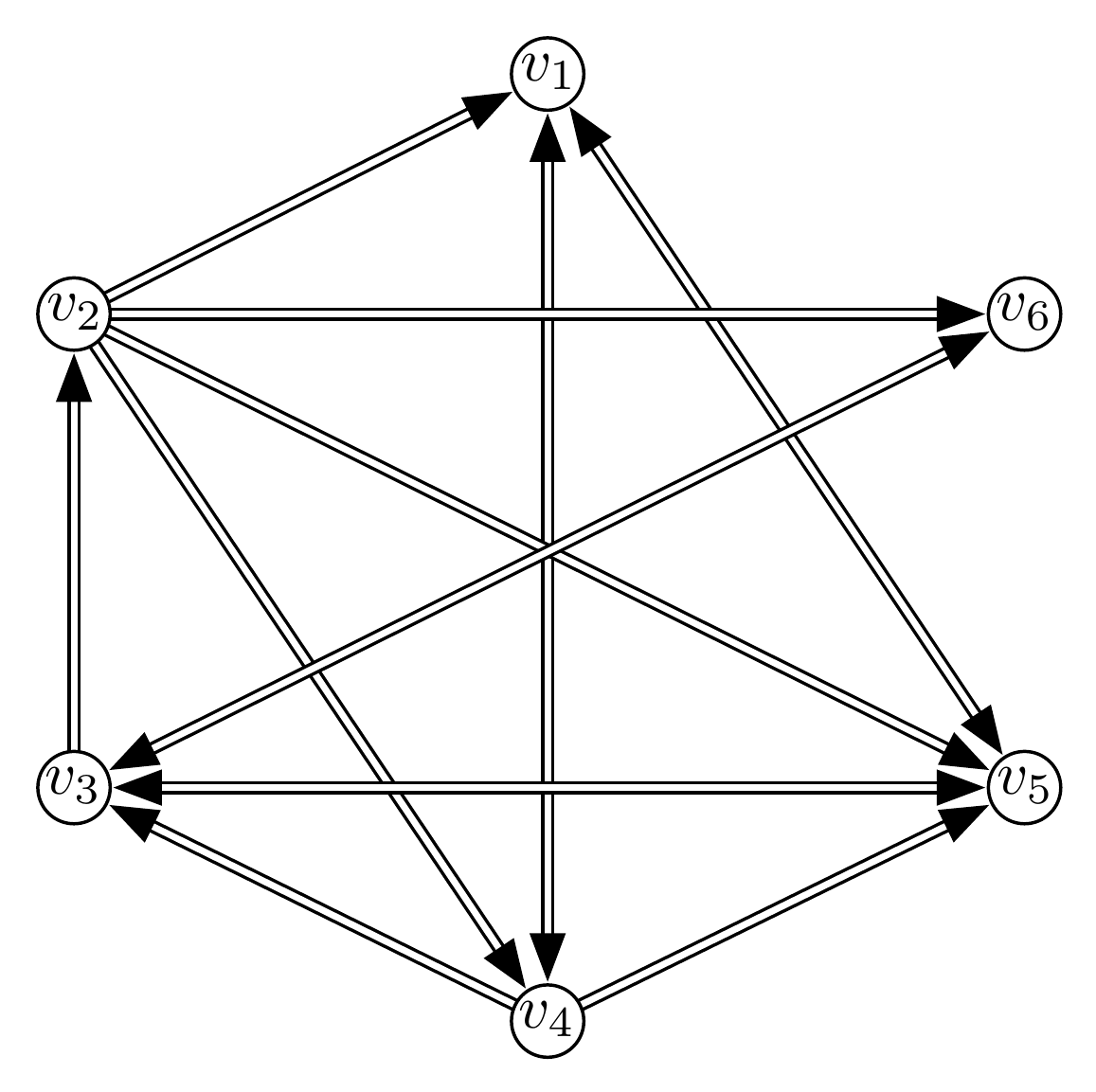}\label{fig:sims2}\quad\quad\quad\quad}\;
\subfigure[]\eB \; \subfigure[]\epB  \; \subfigure[]\trajB
  \caption{Results for the two simulation case studies. Top row: the case of an infinitesimally rigid framework. Bottom row: the case of a non-rigid framework. Note how in the first case (top) the bearing error vector $e(t)$ (Fig.~(b)) and the cumulative position estimation error $e_p(t)$ (Fig.~(d)) correctly converge to $0$. This can also be appreciated in Fig.~(d) where the trajectories of the estimated positions and orientations are shown superimposed to their true values. The results are of course completely different for case~II (bottom) where the estimation errors do not converge to $0$ because of the non-rigidity of the employed framework}\label{fig:sims_plots}
\end{center}
\end{figure*}

\section{Conclusion}\label{sec:conclusion}
This work proposed a distributed estimator for estimating the unscaled relative positions of a team of agents in a common reference frame.  The key feature of this work is the estimation only requires bearing measurements that are expressed in the local frame of each agent.  The estimator builds on a corresponding extension of rigidity theory for frameworks in $SE(2)$.  The main contribution of this work, therefore, was the characterization of infinitesimal rigidity in $SE(2)$.  It was shown that infinitesimal rigidity of the framework is related to the rank of the directed bearing rigidity matrix.  The null-space of that matrix describes the infinitesimal motions of an $SE(2)$ framework, and include the rigid body translations and dilations, in addition to coordinated rotations. 

To our knowledge, this is the first formal characterization of rigidity theory for $SE(2)$ frameworks.  We believe there are many natural and interesting directions for further research, including the development of analogous results from distance and parallel rigidity theory to this setup.  A future work of ours is considering how rigidity in $SE(2)$ can be used to develop distributed control laws from bearing measurements.

{\small
\bibliographystyle{IEEEtran}
\bibliography{./alias,./bibCustom}
}

\end{document}